\documentclass[11pt]{amsart}
\usepackage{amsfonts}
\usepackage{amsmath}
\usepackage{amssymb, amsthm}
\usepackage{bm}
\usepackage{mathtools}
\usepackage{hyperref}
\usepackage{enumerate}
   \allowdisplaybreaks
   
\newcounter{serreinv}

\usepackage{comment}

\mathchardef\period=\mathcode`.
\DeclareMathSymbol{.}{\mathord}{letters}{"3B}
\usepackage{tikz-cd}
\usetikzlibrary{decorations.pathmorphing}

\makeatletter
\newtheorem*{rep@theorem}{\rep@title}
\newcommand{\newreptheorem}[2]{%
\newenvironment{rep#1}[1]{%
 \def\rep@title{#2 \ref{##1}}%
 \begin{rep@theorem}}%
 {\end{rep@theorem}}}
\makeatother
\tikzset{
  symbol/.style={
    draw=none,
    every to/.append style={
      edge node={node [sloped, allow upside down, auto=false]{$#1$}}}
  }
}
\newtheorem{lemma}{Lemma}[section]
\newreptheorem{lemma}{Lemma}
\newtheorem{theorem}[lemma]{Theorem}
\newreptheorem{theorem}{Theorem}
\newtheorem{corollary}[lemma]{Corollary}
\newreptheorem{corollary}{Corollary}
\newtheorem{prop}[lemma]{Proposition}
\newreptheorem{prop}{Proposition}

\theoremstyle{definition}
\newtheorem{defn}[lemma]{Definition}
\newtheorem{rem}[lemma]{Remark}

\theoremstyle{remark}
\newtheorem*{rem*}{Remark}
\newtheorem*{note*}{Note}

\newcommand\restr[2]{{% we make the whole thing an ordinary symbol
  \left.\kern-\nulldelimiterspace % automatically resize the bar with \right
  #1 % the function
  \vphantom{\big|} % pretend it's a little taller at normal size
  \right|_{#2} % this is the delimiter
  }}
  \DeclareMathSymbol{.}{\mathord}{letters}{"3B}
\def\temp{&} \catcode`&=\active \let&=\temp
\begin{document}
\title{Stability conditions on non-commutative curves}
\author{Benjamin Sung}
\address{Department of Mathematics, University of California, Santa Barbara, CA 93106, USA}
\email{bsung@ucsb.edu}
\date{}
\begin{abstract}
%{\color{red} TODO: write introduction, verify Lemma 2.0.2 for the last time, verify computation of 3.4 for the last time, write section 6.5}
We prove that any non-commutative smooth projective variety with a Bridgeland stability condition of dimension less than $\frac{6}{5}$ must be a smooth projective curve. As a consequence, we deduce the non-existence of such categories with dimension in the interval $(1,\frac{6}{5})$. Moreover, we prove a sharp reconstruction result for smooth projective curves of higher genus using the moduli space of stable objects in a category of dimension $1$, and deduce a structural result for their semi-orthogonal decompositions.
%{\color{red} write abstract, go over introduction, think through section 3 one last time}
\end{abstract}
\maketitle
\tableofcontents
\section{Introduction}
The study of D-branes in string theory has undergone a number of refinements, requiring the introduction of increasingly mathematical tools. They were first found to admit an interpretation in topological K-theory in \cite{1998}, and then as an object of the derived category in \cite{kontsevich1994homological} together with a BPS condition refined as $\Pi$-stability in \cite{2001} in the topological sector. Since its inception as a precise mathematical structure in \cite{2002math.....12237B}, the theory of Bridgeland stability conditions has occupied a unique place in mathematical physics; it has become a foundational tool in algebraic geometry for the construction and study of moduli spaces of sheaves and their birational geometry together with a wealth of applications to classical problems. For a survey of D-branes and $\Pi$-stability in the derived category, we refer the reader to the standard surveys~\cite{1999,2005} and for applications of Bridgeland stability to algebraic geometry, we refer to the surveys~\cite{macri2019lecturesbs,macri2019lectures,bayer2022unreasonable}.

A somewhat less studied line of inquiry is the investigation of general properties enjoyed by triangulated categories which admit a Bridgeland stability condition. Results in this direction would complement a growing, pervasive theme in algebraic geometry, namely, that many higher dimensional geometric phenomena are often governed by lower dimensional non-commutative smooth projective varieties together with their moduli spaces of Bridgeland stable objects. 
%\begin{question}\label{q:motivation}
%\begin{enumerate}[(a)]
%\item
%Are there sharp criteria delineating the space of triangulated categories admitting a Bridgeland stability condition from those that provably do not?
%\item
%Does the existence of a Bridgeland stability condition yield a selection mechanism for identifying geometric categories\footnote{i.e. categories that are deformation equivalent to $D^b(X)$ with $X$ a smooth, projective variety}?
%\end{enumerate}
%\end{question}
%{\color{red}application to structure of semiorthogonal decompositions and minimal model}
This leads us to the main subject of this article: Can we classify non-commutative smooth projective varieties and to what extent does the existence of a stability condition constrain their structure?

%This article is dedicated to addressing Question~\ref{q:motivation}. 
Given an arbitrary triangulated category equipped with a stability condition, either a systematic classification or even the deduction of general properties is intractable without further assumptions. %{\color{red}dimensions of categories, yu qiu \cite{qiu2018global}, elagin-lunts}
As an example of the technical complexity involved, there are many possible definitions for even the dimension of a triangulated category, as pointed out in~\cite{2019arXiv190109461E}. Thus, to facilitate our analysis, we will focus on the notion of a global dimension of a stability condition~\cite{qiu2018global}, which roughly bounds the possible difference between phases of semistable objects related by a nonzero morphism.

A first step was undertaken in \cite{Kikuta_2021}, which classified triangulated categories admitting a stability condition $\sigma$ of global dimension $gldim(\sigma) < 1$ as $D^b(Q)$, with $Q$ a Dynkin quiver, and also characterized the global dimension of stability conditions on $D^b(C)$, with $C$ a smooth projective curve. In this article, we will study admissible subcategories $\mathcal{D} \xhookrightarrow{} D^b(X)$ with a stability condition $\sigma$ of global dimension $gldim(\sigma) < 2$. One of the fundamental implications for categories $\mathcal{D}$, admitting a stability condition, is the existence of well-behaved moduli spaces of semistable objects. By leveraging the fundamental results developed and discussed in \cite{macri2019lectures,2021}, we will identify a suitable curve $C$, in the moduli space of semistable objects in $\mathcal{D}$ under appropriate assumptions, and construct an equivalence of categories $\mathcal{D} \simeq D^b(C)$. As a consequence, we deduce the non-existence of such categories with global dimension $gldim(\sigma) < \frac{6}{5}$, deduce a converse to \cite[Theorem 5.16]{Kikuta_2021} and partially extend \cite[Theorem 5.12]{Kikuta_2021}.
%In this section, we prove several results towards the classification of noncommutative curves equipped with a stability condition complementary to the analysis of \cite{Kikuta_2021}. Specifically, we will prove a partial converse to \cite[Theorem 5.16]{Kikuta_2021}.
\subsection{Summary of results} Throughout this paper, we will work over a fixed field $k$, algebraically closed and of characteristic $0$. Our main interest will be in the classification of $k$-linear triangulated categories $\mathcal{D}$ that are {\em connected}, i.e. do not admit completely orthogonal decompositions, and that are {\em geometric non-commutative schemes}~\cite{MR3545926}, i.e. admit an admissible embedding $\mathcal{D} \xhookrightarrow{} D^b(X)$ with $X$ a smooth, projective variety. Our techniques rely heavily on the theory of Bridgeland stability conditions, and we denote by $Stab(\mathcal{D})$ the space of stability conditions. Moreover, we denote by $Stab_{\mathcal{N},d}(\mathcal{D})$ the subspace of stability conditions $\sigma = (\mathcal{A},Z)$ that are {\em numerical}, i.e. $Z$ factors through the numerical Grothendieck group, and such that the image $Im(Z) \subset \mathbb{C}$ is discrete.

In the following, let $\mathcal{D}$ be a connected, geometric non-commutative scheme. Our main theorem asserts that under the assumption that the dimension is less than $\frac{6}{5}$, we can classify all such categories.
%\begin{repprop}{prop:reconstructionA}
%Assume that $\mathcal{D}$ is a non-rational noncommutative curve in the sense of Definition~\ref{defn:noncommcurve} and that there exists a stability condition $\sigma = (\mathcal{A},Z)$ with $gldim(\sigma) <2$ such that $Im(Z) \subset \mathbb{C}$ is discrete. If there exists an object $E \in \mathcal{D}$ such that $Hom^1(E,E) = \mathbb{C}$, then there exists an admissible embedding $\Phi \colon D^b(C) \xhookrightarrow{} \mathcal{D}$ where $C$ is a smooth projective curve. 

%If instead, there exists a $1$-Calabi-Yau object and $\sigma$ is Serre-invariant, then the same conclusion holds. In this case, if $\mathcal{D}$ is also connected, then $\Phi$ is an equivalence of categories.
%\end{repprop}
\begin{reptheorem}{thm:mainthm}
Assume that $\mathcal{D}$ has no exceptional objects and $\inf\limits_{\sigma \in Stab_{\mathcal{N},d}(\mathcal{D})}gldim(\sigma) < \frac{6}{5}$. 
Then there exists an equivalence $\mathcal{D} \simeq D^b(C)$ with $C$ a smooth projective curve of genus $g \geq 1$.
\end{reptheorem}

%We have presented our main result in such a way as to maximize clarity. Indeed, as we point out in remark~\ref{rem:infdim}, we can replace the assumption on the global dimension with the assumption that $\inf\limits_{\sigma \in Stab(\mathcal{D})}gldim(\sigma) < \frac{6}{5}$ and that there exists a separate stability condition $\sigma'$ which is numerical, has discrete central charge, and satisfies $gldim(\sigma') < 2$. 

As a consequence of \cite[Theorem 5.16]{Kikuta_2021}, which deduced the equality $\inf\limits_{\sigma \in Stab_{\mathcal{N},d}(D^b(C))}gldim(\sigma) = 1$ for all smooth projective curves $C$ of genus $g \geq 1$, we obtain the following corollary.

\begin{repcorollary}{cor:nonexistence}
There exists no such $\mathcal{D}$ with no exceptional objects and $\inf\limits_{\sigma\in Stab_{\mathcal{N},d}(\mathcal{D})}gldim(\sigma) \in( 1, \frac{6}{5})$.
\end{repcorollary}

We further build on Theorem~\ref{thm:mainthm} in dimension $1$ to obtain a characterization in terms of the upper Serre dimension, which roughly quantifies the growth of phases of objects under iterated applications of the Serre functor. The following serves as a converse to \cite[Theorem 5.16]{Kikuta_2021}.

\begin{reptheorem}{thm:infdim1}
Assume that $\mathcal{D}$ has no exceptional objects. Assume that there exist a numerical stability condition $\sigma \in Stab(\mathcal{D})$ such that $gldim(\sigma) < 2$ and $Im(Z) \subset \mathbb{C}$ is discrete. Then the following are equivalent:
\begin{enumerate}
\item
$\inf\limits_{\sigma \in Stab(\mathcal{D})}gldim(\sigma) = 1$
\item
$\mathcal{D} \simeq D^b(C)$ with $C$ a smooth projective curve of genus $g \geq 1$.
\setcounter{serreinv}{\value{enumi}}
\end{enumerate}
If in addition, $\mathcal{D}$ admits a Serre-invariant Bridgeland stability condition, then the above is equivalent to the following.
\begin{enumerate}
\setcounter{enumi}{\value{serreinv}}
\item
$\overline{Sdim}(\mathcal{D}) = 1$
\end{enumerate}
\end{reptheorem}

Finally, we specialize to cases with a stability condition $\sigma$ such that $gldim(\sigma) = 1$ and we obtain a characterization of the semi-orthogonal decomposition in the presence of exceptional objects.

\begin{repcorollary}{cor:gldim1}
Assume that there exists a numerical Bridgeland stability condition $\sigma \in Stab(\mathcal{D})$ such that $Im(Z) \subset \mathbb{C}$ is discrete and such that $gldim(\sigma) = 1$. Then $\mathcal{D} = \langle \mathcal{C}, E_1 ,\ldots E_n \rangle$ admits a semi-orthogonal decomposition for some integer $n$ with $E_i \in \mathcal{D}$ exceptional objects and $\mathcal{C}$ is either zero or equivalent to $D^b(E)$, with $E$ a smooth elliptic curve.
\end{repcorollary}

It would be interesting to investigate the sharpness of our results in Theorem~\ref{thm:mainthm}, and whether the upper bound can be generalized to $\frac{4}{3}$. Indeed, the Kuznetsov component $Ku(X) \subset D^b(X)$ with $X$ the cubic surface is fractional Calabi-Yau of dimension $\frac{4}{3}$. Though it is generated by a full exceptional collection, we expect that our results are insensitive to the base field and that the category, $Ku(X)$, for $X$ a Picard rank $1$ cubic surface over $\mathbb{Q}$ is indecomposable and in fact satisfies the equality $\inf\limits_{\sigma \in Stab_{\mathcal{N},d}(Ku(X))}gldim(\sigma) =\frac{4}{3}$.\footnote{We thank Xiaolei Zhao for conversations regarding this example.}

%We expect that our results in Theorem~\ref{thm:mainthm} are in fact sharp. Indeed, the Kuznetsov component $Ku(X) \subset D^b(X)$ with $X$ the cubic surface is fractional Calabi-Yau of dimension $\frac{6}{5}$. On the other hand, it is generated by a full exceptional collection. Nevertheless, we expect that our results are insensitive to the base field and that the category, $Ku(X)$, for $X$ a Picard rank $1$ cubic surface over $\mathbb{Q}$ is indecomposable and satisfies the equality $\inf\limits_{\sigma \in Stab_{\mathcal{N},d}(Ku(X))}gldim(\sigma) =\frac{6}{5}$.\footnote{We thank Xiaolei Zhao for conversations regarding this example.}

\subsection{Related works}
We note that the papers, \cite{MR2427460} and \cite{ctx31462736420006531}, classified abelian categories of homological dimension $1$ without exceptional objects and admitting Serre duality up to derived equivalence, by relying on the general results in~\cite{MR1887637}. In particular, this implies our reconstruction results in the case of $1$-Calabi-Yau categories. On the other hand, our analysis using moduli space techniques is more direct and is ultimately necessary in higher dimensional cases, where there a priori does not exist a heart of bounded t-structure of homological dimension $1$ satisfying Serre duality.
\subsection{Organization}
The organization of this paper is as follows. In Section~\ref{generalprop}, we formulate basic results for $\sigma$-semistable objects in categories satisfying $gldim(\sigma) < 2$. In Section~\ref{sec:1spherical}, we formulate a stronger condition which guarantees the existence of a $1$-spherical object. In Section~\ref{1cyreconstruction}, we prove a geometric reconstruction result using the moduli space of $\sigma$-semistable objects. Finally in Section~\ref{sec:main}, we prove our main theorem and deduce a number of implications in dimension $1$.
%This section is organized as follows. In subSection~\ref{generalprop}, we prove several basic properties for non-rational noncommutative curves. In subSection~\ref{1cyexistence}, we formulate a general criterion to ensure the existence of a $1$-Calabi-Yau object in a general triangulated category with a stability condition. In subSection~\ref{1cyreconstruction}, we prove Proposition~\ref{prop:reconstructionA}. In subSection~\ref{gldim1}, we prove Theorem~\ref{thm:gldim1}. In subSection~\ref{serreinvariant}, we prove Proposition~\ref{prop:reconstructionS}.
\subsection*{Acknowledgements.}
I would like to thank my advisor, Emanuele Macr\`i, for extensive discussions throughout the years. I am grateful to the University of Paris-Saclay for hospitality during the completion of this work. This work was partially supported by the NSF Graduate Research Fellowship under grant DGE-1451070 and by the ERC Synergy Grant ERC-2020-SyG-854361-HyperK.

%\subsection{Stable objects}\label{sec:stable}

\section{Stable objects on non-commutative curves}\label{generalprop}
Throughout this paper, the theory of Bridgeland stability conditions will serve as a fundamental tool. We will abide by the standard terminology: given a stability condition $\sigma = (\mathcal{A},Z)$ on a triangulated category $\mathcal{T}$, $\mathcal{A}$ denotes the corresponding heart of bounded t-structure, $Z \colon K_0(\mathcal{A}) \rightarrow \mathbb{C}$ denotes the central charge, and $P_\sigma$ denotes the associated slicing. Moreover, given an object $E \in \mathcal{T}$, $\varphi_\sigma^+(E)$ and $\varphi_\sigma^-(E)$ denote the maximal and minimal phases respectively, i.e. the phases of the maximal and minimal $\sigma$-semistable objects in the unique Harder-Narasimhan filtration of $E$.

In this section, we will prove several general results for categories behaving similarly to $D^b(C)$, with $C$ a smooth projective curve of genus $g \geq 1$. We first recall the following definition.
\begin{defn}
Let $\mathcal{T}$ be a triangulated category with a Bridgeland stability condition $\sigma = (\mathcal{A},Z)$. The {\em global dimension} of $\sigma$ is defined as follows.
\[
gldim(\sigma) \coloneqq \sup\{\varphi_2 -\varphi_1 \vert Hom_\mathcal{T}(E_1,E_2) \neq 0 \text{ for } E_i \in \mathcal{P}_\sigma(\varphi_i) \} \in [0,+\infty]
\]
\end{defn}
In order to identify with the setting of higher genus curves, we will also define the following.
\begin{defn}\label{defn:noncommcurve}
A triangulated category $\mathcal{T}$ is
\begin{enumerate}
\item
{\em connected} if there does not exist any non-trivial completely orthogonal decompositions.
\item
{\em non-rational} if there exists no exceptional objects in $\mathcal{D}$.
 \end{enumerate}
\end{defn}

In the following, we will always assume that $\mathcal{T}$ is a non-rational triangulated category with Serre functor $S$ and of finite type over $k$, i.e. for any objects $E,F$, the vector space $\bigoplus\limits_i Hom(E,F[i])$ is finite dimensional. In addition, we assume that $\mathcal{T}$ admits a stability condition $\sigma = (\mathcal{A},Z)$ of global dimension $gldim(\sigma) < 2$.

Recall that there exists a {\em Mukai vector} $v \colon Ob(\mathcal{T}) \rightarrow K_{num}(\mathcal{T})$ sending an object $E \in \mathcal{T}$ to its class in the numerical Grothendieck group where $K_{num}(\mathcal{T}) = K_0(\mathcal{T}) / ker(\chi)$ with $$\chi \coloneqq \sum (-1)^i hom^i(-,-)\colon K_0(\mathcal{T}) \times K_0(\mathcal{T})\rightarrow \mathbb{Z}$$ the Euler pairing. For convenience of notation, for any objects $E, F \in \mathcal{D}$, we define the pairing $(v(E), v(F))\coloneqq -\chi(E,F)$. We first observe the following. 
\begin{lemma}\label{lem:semistable}
Let $E \in \mathcal{D}$ be a $\sigma$-semistable object. Then the following holds.
\begin{enumerate}
\item
We have the inequality $v(E)^2 \geq 0$.
\item
If $v(E)^2 = 0$, then every $\sigma$-stable factor $E_i$ satisfies $v(E_i)^2 = 0$.
\item
$Hom(E,SE[-1]) \neq 0$. 
\end{enumerate}
\end{lemma}
\begin{proof}
For the first claim, by definition and the assumption $gldim(\sigma) < 2$, we have the equalities
\[
v(E)^2 = -\chi(E,E) = -hom(E,E) + hom(E, E\left[1\right])
\]
We may assume that $E$ is strictly semistable, otherwise, it must be the case that $Hom(E,E) = k$. Then by the assumption of nonexistence of exceptional objects, we must have $Hom(E,E[1]) \neq 0$ and the claim follows. Otherwise, take a Jordan-H\"{o}lder filtration of $E$ with $\sigma$-stable factors $E_i$ of the same phase. Then we have $v(E) = \sum_i v(E_i)$ and in particular, the equality
\begin{equation}\label{eq:vsquare}
v(E)^2 = \sum_i v(E_i)^2 + \sum_{i\neq j} (v(E_i),v(E_j))
\end{equation}
It suffices to prove that $\sum_{i\neq j} (v(E_i),v(E_j)) \geq 0$. By definition, recall that 
\[(v(E_i),v(E_j)) = -hom(E_i,E_j) + hom(E_i,E_j\left[1\right])
\]
If $E_i$ and $E_j$ were isomorphic, then we immediately have $(v(E_i),v(E_j)) \geq 0$ for all $i,j$. But if they were not, then we must have $Hom(E_i,E_j) = 0$ and we conclude.

For the second claim, assume that $v(E)^2 = 0$. Then by equation~\eqref{eq:vsquare}, together with the preceding paragraph, every term in the sum on the rhs of equation~\eqref{eq:vsquare} must be $\geq 0$ and therefore must be $0$. In particular, we must have $v(E_i)^2 = 0$ for all $i$ and we conclude.

For the third claim, assume on the contrary that $Hom(E,SE[-1]) = 0$. Then by Serre duality, we have $Hom(E,E[1]) = 0$. But then it must be the case that $v(E)^2 < 0$, contradicting the first claim.
\end{proof}

Next, we will reformulate the well-known weak Mukai Lemma in our setting using the assumption on the global dimension. 
\begin{lemma}[Weak Mukai Lemma]\label{lem:wml}
Let $A \rightarrow E \rightarrow B$ be an exact triangle in $\mathcal{T}$ satisfying the inequality $\varphi_\sigma^-(A) > \varphi_\sigma^+(B)$.
Then we have the following inequality:
\[
hom^1(A,A) + hom^1(B,B) \leq hom^1(E,E)
\]
\end{lemma}
\begin{proof}
It suffices to prove the two vanishings $Hom(B, A[2]) = 0, Hom(A,B) = 0$. The conclusion then follows as in the proof of \cite[Lemma 2.5]{2017}. To see the first claim, we observe the following:
\[
\varphi_\sigma^-(A[2]) = \varphi_\sigma^-(A) + 2 > \varphi_\sigma^+(B) +2
\]
where the first equality follows by uniqueness of the Harder-Narasimhan filtration of $A$ into $\sigma$-semistable factors and the second inequality follows by assumption. In particular, this implies that $\varphi_\sigma(A_i[2]) > \varphi_\sigma(B_j) +2$ for any $\sigma$-semistable factor $A_i[2], B_j$ of $A[2],B$ respectively. But this implies that $Hom(B_j, A_i[2]) = 0$ for all $i, j$ by the assumption that $gldim(\sigma) < 2$. In particular, this implies that $Hom(B,A[2]) = 0$. The second claim follows from an analogous argument using the inequality $\varphi_\sigma^-(A) > \varphi_\sigma^+(B)$ and we conclude.
\end{proof}
We conclude this section by studying objects $E$ with minimal $hom^1(E,E)$, obtaining a certain converse statement to Lemma~\ref{lem:semistable}.
\begin{lemma}\label{lem:1cy}
Let $d$ be the minimum integer such that there exists an object $D \in \mathcal{T}$ with $d = hom^1(D,D)$.
\begin{enumerate}
\item
The integer $d$ is strictly positive.
\item
If $d \geq 2$, then any object $E$ such that $hom^1(E,E) \leq 2d -2$ is $\sigma$-stable.
\item
If $d = 1$, then any object $E$ such that $hom^1(E,E) = 1$ is $\sigma$-semistable and $v(E)^2 = 0$.
 \end{enumerate}
\end{lemma}
\begin{proof}
(1): Assume on the contrary that $d = 0$. Let $E$ be an object such that $Hom^1(E,E) = 0$. By Lemma~\ref{lem:wml}, we may assume that $E$ is $\sigma$-semistable. Indeed if not, then taking the Harder-Narasimhan filtration, Lemma~\ref{lem:wml} implies that every $\sigma$-semistable factor $E_i$ must satisfy $Hom^1(E_i,E_i) = 0$. On the other hand, this contradicts Lemma~\ref{lem:semistable}(1).

(2): We first prove that if $d > 0$ and $E$ satisfies $hom^1(E,E) \leq 2d-1$, then $E$ must be $\sigma$-semistable. Assume not. Then the Harder-Narasimhan filtration yields $n$ $\sigma$-semistable factors $E_i$. In particular, we have the inequalities
\[
nd \leq \sum\limits_{i = 1}^nhom^1(E_i,E_i) \leq hom^1(E,E) \leq 2d -1
\]
where the first follows from the minimality of $d$, the second follows from Lemma~\ref{lem:wml}, and the third follows by assumption. Clearly, we must have $n = 1$. 

We now prove that if $d \geq 2$ and $E$ satisfies $hom^1(E,E) \leq 2d-2$, then $E$ is also $\sigma$-stable. Taking a Jordan-Holder filtration into $m$ $\sigma$-stable factors $F_i$ of the same phase, we obtain the equality
\begin{equation*}
v(E)^2 = \sum_i v(F_i)^2 + \sum_{i\neq j} (v(F_i),v(F_j))
\end{equation*}
As in the proof of the first paragraph of Lemma~\ref{lem:semistable}, we must have $ (v(F_i),v(F_j)) \geq 0$ for all non-equal pairs $i,j$. We obtain the inequalities
\[
md - m \leq \sum\limits_{i = 1}^m hom^1(F_i,F_i) -m  = \sum\limits_{i=1}^m v(F_i)^2 \leq v(E)^2 \leq (2d-2) - hom(E,E) \leq 2d-3
\]
where the first follows from the minimality of $d$, the second is by definition, the third follows from the preceding paragraph, the fourth is by definition, and the fifth follows from the inequality $hom(E,E) \geq 1$. On the other hand, this again implies $m=1$, and we conclude.

(3): The fact that $E$ is $\sigma$-semistable follows from the first paragraph of the proof of part $(2)$. The claim that $v(E)^2 = 0$ follows from the fact that $hom(E,E) \geq 1$ and Lemma~\ref{lem:semistable}(1).
\end{proof}

\section{Global dimension bounds and $1$-spherical objects}\label{sec:1spherical}
Let $\mathcal{T}$ be a triangulated category of finite type over $k$ with a Serre functor $S$, admitting a Bridgeland stability condition, i.e. $Stab(\mathcal{D}) \neq 0$. In the following, let $d$ be the minimum integer such that there exists $E \in \mathcal{D}$ with $d = Hom^1(E,E)$. The goal of this section is to prove the following, largely following the methods exhibited in \cite[Section 3]{ctx31462736420006531}.
\begin{prop}\label{prop:1sphericalnew}
Assume $\mathcal{T}$ satisfies $d \geq 1$ and $\inf\limits_{\sigma \in Stab(\mathcal{T})} gldim(\sigma) < \frac{6}{5}$. Then $d = 1$ and every object $F \in \mathcal{T}$ satisfying $hom^1(F,F) = 1$ is $1$-spherical. 
\end{prop}
\subsection{Bounds on extensions of objects with minimal $\bm{hom^1}$}\label{sec:minimal} In this subsection, we introduce a stronger condition than the existence of a stability condition $\sigma$ satisfying $gldim(\sigma) <2$, summarized in Definition~\ref{defn:noncurve}. Under this assumption, we deduce a number of restrictions on objects $E$ with minimal $hom^1(E,E)$. Most notably, we demonstrate in Lemma~\ref{lem:1spherical} that if $d = 1$, then any such object must be $1$-spherical. In Lemma~\ref{lem:restrictC}, we derive bounds on the object $Cone(E \rightarrow SE[-1])$ in the case $d \geq 2$. 

We first introduce the following definition which refines the assumption on the existence of a stability condition $\sigma$ satisfying $gldim(\sigma) <2$.
\begin{defn}\label{defn:noncurve}
We will say that $\mathcal{T}$ is a {\em noncommutative curve} if there exists a stability condition $\sigma = (\mathcal{A} = \mathcal{P}(\varphi_0,\varphi_0+1],Z)$ such that
\begin{enumerate}
\item
There exists an object $E \in \mathcal{T}$ with $hom^1(E,E) = d$ such that $E, SE[-1] \in \mathcal{A}$. 
\item
We have the inequalities: $$(\varphi_0 + 2) - \varphi(SE[-1]) > gldim(\sigma),\qquad (\varphi(E) + 2) - (\varphi_0 + 1) > gldim(\sigma)$$
\end{enumerate}
\end{defn}
\begin{rem}\label{rem:noncurvedim2}
If $\mathcal{T}$ is a noncommutative curve, then Definition~\ref{defn:noncurve}(2) immediately implies that the stability condition $\sigma$ satisfies $gldim(\sigma) < 2$. In addition, if $d \geq 1$, then $\mathcal{T}$ is also non-rational and the results in Section~\ref{generalprop} apply.
\end{rem}

In the following, $\mathcal{T}$ will always be a non-commutative curve such that $d \geq 1$. We fix an object $E \in \mathcal{T}$ and a stability condition $\sigma = (\mathcal{A} = \mathcal{P}(\varphi_0,\varphi_0+1],Z)$ satisfying Definition~\ref{defn:noncurve}. By Lemma~\ref{lem:1cy}, we note that $E$ and $SE[-1]$ are semistable with respect to any stability condition $\sigma$ with $gldim(\sigma) < 2$. Let $f \in Hom(E,SE[-1]) = Hom^1(E,E) \neq 0$ be any nonzero morphism and consider the exact triangle.
\[
\begin{tikzcd}
E \arrow[r, "f"] & SE[-1] \arrow[r] & C
\end{tikzcd}
\]
Taking cohomology with respect to $\mathcal{A}$, we have the following exact sequences
\begin{equation}\label{eq:exact1}
\begin{tikzcd}
0 \arrow[r] &\mathcal{H}^{-1}(C) \arrow[r] & E \arrow[r] & im(f) \arrow[r]& 0 
\end{tikzcd}
\end{equation}
\begin{equation}\label{eq:exact2}
\begin{tikzcd}
0 \arrow[r] & im(f) \arrow[r]& SE[-1] \arrow[r] &  \mathcal{H}^0(C) \arrow[r] &0 
\end{tikzcd}
\end{equation}
We begin by recording the following general Lemmas which will be used extensively later in establishing Proposition~\ref{prop:1sphericalnew}. We first study some general bounds imposed on phases of the object $C$.
\begin{lemma}\label{lem:boundC}
The objects $E, SE[-1]$ are $\sigma$-semistable. If $\mathcal{H}^{-1}(C), \mathcal{H}^0(C) \neq 0$, then we have the following chain of inequalities on the phases.
\begin{align*}
\varphi_0 &< \varphi^-(\mathcal{H}^{-1}(C)) \leq \varphi^+(\mathcal{H}^{-1}(C)) \leq \varphi(E) \\
&\leq \varphi^-(im(f)) \leq \varphi^+(im(f)) \leq \varphi(SE[-1])\\ &\leq \varphi^-(\mathcal{H}^0(C)) \leq \varphi^+(\mathcal{H}^0(C)) \leq \varphi_0+1
\end{align*}
\end{lemma}
\begin{proof}
The claim that $E,SE[-1]$ are $\sigma$-semistable follows directly from the proof of Lemma~\ref{lem:1cy}(2). The first and last inequalities follow from the assumption that all objects are contained in $\mathcal{A}$. The second, fifth, and eighth inequalities are clear by definition. The rest of the inequalities follow directly from \cite[Lemma 3.4]{2002math.....12237B}.
\end{proof}
\begin{lemma}\label{lem:simple}
$E$ is a simple object, i.e. $hom(E,E) = 1$. 
\end{lemma}
\begin{proof}
If $d = 1$, then the statement follows from Lemma~\ref{lem:semistable}(1) and the fact that $E$ is $\sigma$-semistable from Lemma~\ref{lem:boundC} . If $d \geq 2$, then this follows directly by Lemma~\ref{lem:1cy}(2) as $E$ is $\sigma$-stable.
\end{proof}
Finally, we record the following calculation which will be used extensively later, which highlights the utility of Definition~\ref{defn:noncurve}.
%\begin{lemma}\label{lem:longsequence}
%We have the following equalities.
%\begin{align*}
%Hom^1(E,\mathcal{H}^{-1}(C)) &= d - 1 + hom(E, im(f)) - hom^1(E, im(f)) \\
%Hom^1(\mathcal{H}^0(C), SE[-1]) &= d - 1 - hom(E, im(f)) + hom^1(E, im(f)) 
%\end{align*}
%\end{lemma}
%\begin{proof}
%We first observe that $hom(E,E) = 1$. If $d = 1$, then $E$ is $\sigma$-semistable by Lemma~\ref{lem:1cy} and the statement follows from Lemma~\ref{lem:semistable}(1). If $d \geq 2$, then this follows directly by Lemma~\ref{lem:1cy} as $E$ is $\sigma$-stable.

%We now claim the vanishings $Hom(E, \mathcal{H}^{-1}(C))= 0$ and $Hom^2(E, \mathcal{H}^{-1}(C)) = 0$. The first equality then follows by applying $Hom(E,\cdot)$ to sequence~\eqref{eq:exact1}. For the first claim, assume $Hom(E, \mathcal{H}^{-1}(C)) \neq 0$. Then the composition with the injection $\mathcal{H}^{-1}(C) \hookrightarrow E$ yields a nonzero morphism which is not an isomorphism, contradicting the fact that $hom(E,E) = 1$. The second claim follows from the fact that $\mathcal{H}^{-1}(C), SE[-1] \in \mathcal{A}$ and thus $Hom^2(E, \mathcal{H}^{-1}(C)) = Hom(\mathcal{H}^{-1}(C)[1], SE[-1]) = 0$.

%The second equality follows from an analogous argument by applying $Hom(\cdot, SE[-1])$ to sequence~\eqref{eq:exact2} and noting that
%\begin{align*}
%Hom(\mathcal{H}^0(C), SE[-1]) &= 0 \\
%Hom^2(\mathcal{H}^0(C), SE[-1]) &= Hom(E[1], \mathcal{H}^0(C)) = 0
%\end{align*}
%\end{proof}

\begin{lemma}\label{lem:mainbound}
We have the following inequalities.
\begin{align*}
hom^1(\mathcal{H}^{-1}(C),\mathcal{H}^{-1}(C)) &\leq hom^1(E,\mathcal{H}^{-1}(C)) = d - 1 + hom(E, im(f)) - hom^1(E, im(f)) \\
hom^1(\mathcal{H}^0(C),\mathcal{H}^0(C)) &\leq hom^1(\mathcal{H}^0(C), SE[-1]) = d - 1 - hom(E, im(f)) + hom^1(E, im(f)) 
\end{align*}
\end{lemma}
\begin{proof}
We first establish the following equalities.
\begin{align*}
hom^1(E,\mathcal{H}^{-1}(C)) &= d - 1 + hom(E, im(f)) - hom^1(E, im(f)) \\
hom^1(\mathcal{H}^0(C), SE[-1]) &= d - 1 - hom(E, im(f)) + hom^1(E, im(f)) 
\end{align*}
We claim the vanishings $Hom(E, \mathcal{H}^{-1}(C))= 0$ and $Hom^2(E, \mathcal{H}^{-1}(C)) = 0$. The first equality then follows by applying $Hom(E,\cdot)$ to sequence~\eqref{eq:exact1}. For the first claim, assume $Hom(E, \mathcal{H}^{-1}(C)) \neq 0$. Then the composition with the injection $\mathcal{H}^{-1}(C) \hookrightarrow E$ yields a nonzero morphism which is not an isomorphism, contradicting the fact that $hom(E,E) = 1$ by Lemma~\ref{lem:simple}. The second claim follows from the fact that $\mathcal{H}^{-1}(C), SE[-1] \in \mathcal{A}$ and thus $Hom^2(E, \mathcal{H}^{-1}(C)) = Hom(\mathcal{H}^{-1}(C)[1], SE[-1]) = 0$. The second equality follows from an analogous argument by applying $Hom(\cdot, SE[-1])$ to sequence~\eqref{eq:exact2} and noting that $Hom(\mathcal{H}^0(C), SE[-1]) = 0$ and $Hom^2(\mathcal{H}^0(C), SE[-1]) = Hom(E[1], \mathcal{H}^0(C)) = 0$.

We now establish the inequalities
\begin{align*}
hom^1(\mathcal{H}^{-1}(C),\mathcal{H}^{-1}(C)) &\leq hom^1(E,\mathcal{H}^{-1}(C))\\
hom^1(\mathcal{H}^0(C),\mathcal{H}^0(C)) &\leq hom^1(\mathcal{H}^0(C),SE[-1])
\end{align*}
which, together with the above, implies the result. We claim the vanishings $Hom^2(im(f), \mathcal{H}^{-1}(C)) = 0$ and $Hom^2(\mathcal{H}^0(C) , im(f)) = 0$. The inequalities then follow by applying $Hom(\cdot,\mathcal{H}^{-1}(C))$ to sequence~\eqref{eq:exact1} and $Hom(\mathcal{H}^0(C),\cdot)$ to sequence~\eqref{eq:exact2}, respectively. The first claim follows from the inequalities
\[
(\varphi^-(\mathcal{H}^{-1}(C)) + 2) - \varphi^+(im(f)) \geq (\varphi_0 + 2) - \varphi(SE[-1]) > gldim(\sigma)
\]
by applying Definition~\ref{defn:noncurve}(2) and Lemma~\ref{lem:boundC}. Similarly, the second claim follows from the inequalities
\[
(\varphi^-(im(f)) + 2) - \varphi^+(\mathcal{H}^0(C)) \geq (\varphi(E) + 2) - (\varphi_0 + 1) > gldim(\sigma)
\]
\end{proof}

We now proceed to the critical Lemmas in establishing Proposition~\ref{prop:1sphericalnew}. We first demonstrate that it suffices to reduce to the case $d = 1$.
\begin{lemma}\label{lem:1spherical}
Assume that $d = 1$. Then $E$ is a $1$-spherical object.
\end{lemma}
\begin{proof}
From the injection $im(f) \hookrightarrow SE[-1]$, the surjection $E \twoheadrightarrow im(f)$, and the assumption $hom(E,SE[-1]) = hom^1(E,E) = 1$, we obtain the inequalities
\begin{align*}
1 &\leq hom(E, im(f)) \leq hom(E,SE[-1]) = 1 \\
1 &\leq hom(im(f),SE[-1]) \leq hom(E,SE[-1]) = 1
\end{align*}
and hence, we have $hom(E,im(f)) = hom^1(E,im(f)) = 1$.

We may assume that $f$ is not an isomorphism, otherwise the conclusion follows immediately. Then it must be the case that either $\mathcal{H}^{-1}(C)$ or $\mathcal{H}^0(C)$ is nonzero. Assume the former case. By Lemma~\ref{lem:mainbound}, we have the following inequality
%\begin{equation*}
%Hom^1(E,\mathcal{H}^{-1}(C)) = d - 1 + hom(E, im(f)) - hom^1(E, im(f))
%\end{equation*}
%Together with the fact that $hom(E,E) = hom^1(E,E) = 1$ and the results of the first paragraph, this implies that $hom^1(E,\mathcal{H}^{-1}(C)) = 0$. On the other hand, by Lemma~\ref{lem:vanishings}, we have $Hom^2(im(f),\mathcal{H}^{-1}(C)) = 0$.
%
%we claim that $Hom^2(im(f),\mathcal{H}^{-1}(C)) = 0$. Indeed, we have 
%\[
%(\varphi^-(\mathcal{H}^{-1}(C)) + 2) - \varphi^+(im(f)) \geq (\varphi_0 + 2) - \varphi(SE[-1]) > gldim(\sigma)
%(\varphi^-(im(f)) + 2)- \varphi^+(\mathcal{H}^{-1}(C)) \geq (\varphi_0 +2) - \varphi^+(\mathcal{H}^{-1}(C)) \geq (\varphi_0+2) - \varphi(SE[-1]) > gldim(\sigma)
%\]
%by applying Lemma~\ref{lem:boundC} and Definition~\ref{defn:noncurve}(2). 
%Thus, applying $Hom(\cdot,\mathcal{H}^{-1}(C))$ to sequence~\eqref{eq:exact1}, we have 
\[hom^1(\mathcal{H}^{-1}(C),\mathcal{H}^{-1}(C)) \leq d - 1 + hom(E, im(f)) - hom^1(E, im(f)) = 0
\]
and so $hom^1(\mathcal{H}^{-1}(C),\mathcal{H}^{-1}(C)) = 0$, contradicting the assumption that $d = 1$ is minimal.

The proof for the case that $\mathcal{H}^0(C)$ is nonzero follows from an identical argument.
\end{proof}
To address the case $d \geq 2$, we prove that the object $C$ must satisfy stronger bounds.
\begin{lemma}\label{lem:restrictC}
Assume that $d \geq 2$. Then exactly one of the objects $\mathcal{H}^{-1}(C)$ or $\mathcal{H}^0(C)$ is nonzero. In addition, $C$ must be $\sigma$-stable.
\end{lemma}
\begin{proof} 
%We first observe that by Lemma~\ref{lem:1cy}, the objects $E, SE[-1]$ are $\sigma$-stable.
%
%By Lemma~\ref{lem:longsequence}, we have the equalities
%\begin{align*}
%Hom^1(E,\mathcal{H}^{-1}(C)) &= d - 1 + hom(E, im(f)) - hom^1(E, im(f)) \\
%Hom^1(\mathcal{H}^0(C), SE[-1]) &= d - 1 - hom(E, im(f)) + hom^1(E, im(f))
%\end{align*}
%
%We now claim the equalities
%\begin{align*}
%Hom^2(im(f), \mathcal{H}^{-1}(C)) &= 0\\
%Hom^2(\mathcal{H}^0(C) , im(f)) &= 0
%\end{align*}
%
%Applying $Hom(\cdot,\mathcal{H}^{-1}(C))$ to sequence~\eqref{eq:exact1} and $Hom(\mathcal{H}^0(C), \cdot)$ to sequence~\eqref{eq:exact2}, the claimed vanishings imply the following bounds
If both objects were zero, then $f$ must be an isomorphism. In particular, we must have $hom^1(E,E) = hom(E,SE[-1]) = hom(E,E) = 1$ by Lemma~\ref{lem:simple}, contradicting the assumption that $d \geq 2$.

If both objects were non-zero, then by Lemma~\ref{lem:mainbound}, we have the following inequalities
\begin{align*}
hom^1(\mathcal{H}^{-1}(C), \mathcal{H}^{-1}(C)) &\leq  d-1 + \chi(E,im(f))\\
hom^1(\mathcal{H}^{0}(C), \mathcal{H}^{0}(C)) &\leq  d-1 - \chi(E,im(f))
\end{align*}
The minimality of $d$ then implies that either $\mathcal{H}^{-1}(C)$ or $\mathcal{H}^0(C)$ must be trivial.

We may consider the case where $\mathcal{H}^{-1}(C) = 0$, the latter case follows from an identical argument. Sequence~\eqref{eq:exact2} then yields the following exact sequence in $\mathcal{A}$.
\begin{equation}\label{eq:exact3}
\begin{tikzcd}
0 \arrow[r] & E \arrow[r] & SE[-1] \arrow[r] & \mathcal{H}^0(C) \arrow[r] & 0
\end{tikzcd}
\end{equation}
%Applying $Hom(\mathcal{H}^0(C), \cdot)$, 
%\[
%\begin{tikzcd}
%&0 = Hom(\mathcal{H}^0(C),SE[-1]) \arrow[r]&Hom(\mathcal{H}^0(C), \mathcal{H}^0(C)) \arrow[lld] \\Hom^1(\mathcal{H}^0(C), E) \arrow[r]& Hom^1(\mathcal{H}^0(C), SE[-1]) \arrow[r] & Hom^1(\mathcal{H}^0(C),\mathcal{H}^0(C)) \arrow[lld] \\ Hom^2(\mathcal{H}^0(C),E) = 0 && 
%\end{tikzcd}
%\]
%\begin{equation}\label{eq:exact6}
%0 \rightarrow Hom(\mathcal{H}^0(C), \mathcal{H}^0(C)) \rightarrow Hom^1(\mathcal{H}^0(C), E) \rightarrow Hom^1(\mathcal{H}^0(C), SE[-1]) \rightarrow Hom^1(\mathcal{H}^0(C),\mathcal{H}^0(C)) \rightarrow 0 
%\end{equation}
%we first observe $Hom^2(\mathcal{H}^0(C), E) = 0$ by applying Definition~\ref{defn:noncurve} and Lemma~\ref{lem:boundC} to obtain the inequalities
%\[
%(\varphi(E) + 2) - \varphi^+(\mathcal{H}^0(C)) \geq (\varphi(E) + 2) - (\varphi_0 +1) > gldim(\sigma)
%\] 
By Lemma~\ref{lem:mainbound}, we have the inequality
\begin{equation}\label{eq:3}
hom^1(\mathcal{H}^0(C),\mathcal{H}^0(C)) \leq hom^1(\mathcal{H}^0(C),SE[-1]) =  d - 1 - \chi(E, E) = 2d-2
\end{equation}
%where the middle equality follows from equation~\eqref{eq:2} in the proof of Lemma~\ref{lem:restrictC}. 
In particular, this implies that $C = \mathcal{H}^0(C)$ is $\sigma$-stable by Lemma~\ref{lem:1cy}. 
\end{proof}
\subsection{Applications of global dimension bounds}
In this subsection, we prove Proposition~\ref{prop:1sphericalnew}. We begin by reformulating Definition~\ref{defn:noncurve} in terms of the quantity $\inf\limits_{\sigma \in Stab(\mathcal{T})}gldim(\sigma)$.

In the following, we assume that $\mathcal{T}$ satisfies $d \geq 1$ and $\inf\limits_{\sigma \in Stab(\mathcal{T})} gldim(\sigma) < \frac{6}{5}$. Fix a stability condition $\sigma = (\mathcal{P},Z)$ such that $gldim(\sigma) < \frac{6}{5}$ and an object $E \in \mathcal{T}$ with $hom^1(E,E) = d$. By Lemma~\ref{lem:1cy}, we again note that $E$ and $SE[-1]$ are semistable with respect to any stability condition $\sigma$ with $gldim(\sigma) < 2$. Consider the following phase
\begin{equation}\label{eq:phi0}
\varphi_0 \coloneqq \frac{1}{2}(\varphi(E) + \varphi(SE[-1])) - \frac{1}{2}
\end{equation}
and the heart $\mathcal{A} = \mathcal{P}(\varphi_0,\varphi_0+1]$. 

We first observe the following simple Lemma, which will be used repeatedly in the below.
\begin{lemma}\label{lem:sbound}
Let $F, SF[-1]\in \mathcal{T}$ be $\sigma$-semistable objects. Then the following inequalities hold.
\[0\leq\varphi(SF[-1]) - \varphi(F) \leq gldim(\sigma)-1
\]
\end{lemma}
\begin{proof}
The first inequality follows immediately from $hom(F,SF[-1]) = hom^1(F,F) \geq 1$ by the assumption on $d$. The second follows from $hom(F,SF) = hom(F,F) \neq 0$ and the definition of the global dimension.
\end{proof}
We now demonstrate that the results of section~\ref{sec:minimal} apply automatically.
\begin{lemma}\label{lem:65noncurve}
$\mathcal{T}$ is a non-commutative curve with respect to $E$ and the heart $\mathcal{A}= \mathcal{P}(\varphi_0,\varphi_0+1]$ in the sense of Definition~\ref{defn:noncurve}.
\end{lemma}
\begin{proof}
It suffices to show that the heart $\mathcal{A} = \mathcal{P}(\varphi_0, \varphi_0+1]$ satisfies the conditions in Definition~\ref{defn:noncurve}.

We first prove that $E,SE[-1] \in \mathcal{A}$. It suffices to prove the inequalities 
\[\varphi_0 <\varphi(E) \leq \varphi(SE[-1]) \leq \varphi_0+1
\]
Applying Lemma~\ref{lem:sbound}, we must have
\begin{equation}\label{eq:sbound}
\varphi(E) \leq \varphi(SE[-1]) \leq \varphi(E) + \frac{1}{5}
\end{equation}
Applying the inequality~\eqref{eq:sbound} to Definition~\ref{eq:phi0}, we obtain
\begin{equation}\label{eq:chain}
\varphi_0 \leq \varphi(E) - \frac{2}{5} < \varphi(E) \leq \varphi(SE[-1]) \leq \varphi(E) + \frac{1}{2} = \frac{1}{2}(\varphi(E) + \varphi(E)) + \frac{1}{2} \leq \varphi_0 + 1
\end{equation}
which implies the claim.

Finally, we verify the inequalities:
\begin{equation}\label{eq:noncurveineq}
(\varphi_0 + 2) - \varphi(SE[-1]) > gldim(\sigma),\qquad (\varphi(E) + 2) - (\varphi_0 + 1) > gldim(\sigma)
\end{equation}
The first follows from the inequalities
\begin{align*}
(\frac{1}{2}(\varphi(E) + \varphi(SE[-1])) + \frac{3}{2})- \varphi(SE[-1]) &= \frac{1}{2}(\varphi(E) - \varphi(SE[-1])) + \frac{3}{2} \\
& \geq \frac{1}{2}(-\frac{1}{5}) + \frac{3}{2} \geq \frac{6}{5} > gldim(\sigma)
\end{align*}
where we used the second inequality in~\eqref{eq:sbound} in the second line. The second inequality in~\eqref{eq:noncurveineq} follows from an identical argument.
\end{proof}
\begin{rem}
We note that Lemma~\ref{lem:65noncurve} applies with an identical argument if $gldim(\sigma) < \frac{4}{3}$.
\end{rem}
Finally, with the stronger bound on the global dimension, we now prove Proposition~\ref{prop:1sphericalnew} by applying Lemmas~\ref{lem:1spherical} and~\ref{lem:restrictC}.
\begin{proof}[Proof of Proposition~\ref{prop:1sphericalnew}]
We will assume that $d \geq 2$. Fix a stability condition $\sigma = (\mathcal{P},Z)$ such that $gldim(\sigma) < \frac{6}{5}$, an object $E$ with $hom^1(E,E) = d$, and heart $\mathcal{A} = \mathcal{P}(\varphi_0,\varphi_0+1]$ with $\varphi_0$ given in~\eqref{eq:phi0}.

Consider the following sequence, where we recall $E,SE[-1] \in \mathcal{A}$ by Lemma~\ref{lem:65noncurve}.
\begin{equation}\label{eq:exact6}
\begin{tikzcd}
E \arrow[r] & SE[-1] \arrow[r] & C 
\end{tikzcd}
\end{equation}
By Lemma~\ref{lem:restrictC}, we may also assume that $C \in \mathcal{A}$ and is $\sigma$-stable. The case with $C \in \mathcal{A}[1]$ will follow by an identical, dual argument.

We define the rotated slicing $\mathcal{A}' \coloneqq \mathcal{P}(\varphi_0+\frac{1}{5}, \varphi_0 + \frac{6}{5}]$. We first observe the inclusion $E,SE[-1],C \in \mathcal{A}'$. Indeed, we have the following chain of inequalities
\[
\varphi_0 + \frac{1}{5} \leq \varphi(E) - \frac{1}{5} < \varphi(E) \leq \varphi(SE[-1]) \leq \varphi(C) \leq \varphi_0 + 1 \leq \varphi_0 + \frac{6}{5}
\]
where the first follows from the first inequality in the chain~\eqref{eq:chain}, the third and fourth by Lemma~\ref{lem:boundC}, and the fifth from the inclusion $C \in \mathcal{A}$.

Applying $Hom(E,\cdot)$ to sequence~\eqref{eq:exact6}, and using that $hom(E,E) =1$ by Lemma~\ref{lem:simple}, it must be the case that $Hom(E,C) \neq 0$. We claim that this implies that $hom(E, SC[-1]) = hom^1(C, E) \geq d $. Indeed, let $g \in Hom(E,C)$ be non-zero. This yields the following exact sequences in $\mathcal{A}'$:
\begin{equation}\label{eq:exact5}
\begin{tikzcd}
0 \arrow[r]& ker(g) \arrow[r] & E \arrow[r] & im(g) \arrow[r] & 0
\end{tikzcd}
\end{equation}
\begin{equation}\label{eq:exact4}
\begin{tikzcd}
0 \arrow[r]& im(g) \arrow[r] & C\arrow[r] & coker(g) \arrow[r] & 0
\end{tikzcd}
\end{equation}
We apply $Hom(im(g), \cdot)$ to the Serre dual of sequence~\eqref{eq:exact4}
\[
\begin{tikzcd}
S(coker(g))[-2] \arrow[r] & S(im(g))[-1] \arrow[r] & SC[-1] \arrow[r] & S(coker(g))[-1] 
\end{tikzcd}
\]
and observe that $Hom(im(g),S(coker(g))[-2]) = Hom(coker(g),im(g)[2]) = 0$ from the inequalities 
\begin{equation}\label{eq:a'bound}
(\varphi^-(im(g)) + 2) - \varphi^+(coker(g)) \geq (\varphi(E) + 2) - (\varphi_0 + \frac{6}{5}) > gldim(\sigma)
\end{equation}
where $\varphi(E) \leq \varphi^-(im(g))$ from \cite[Lemma 3.4]{2002math.....12237B} and $\varphi^+(coker(g)) \leq \varphi_0 + \frac{6}{5}$ as $coker(g) \in \mathcal{A}'$. The second inequality in~\eqref{eq:a'bound} then follows from the inequalities:
\begin{align*}
(\varphi(E) + 2) - (\varphi_0 + \frac{6}{5}) &= \frac{1}{2}(\varphi(E) - \varphi(SE[-1])) + 2 - \frac{7}{10}\\
&\geq \frac{1}{2}(-\frac{1}{5}) + 2 - \frac{7}{10}  = \frac{6}{5} > gldim(\sigma) 
\end{align*}
where the second line follows from Lemma~\ref{lem:sbound}.
Thus, we have $d \leq Hom(im(g), S(im(g))[-1]) \hookrightarrow Hom(im(g), SC[-1])$, and so $d \leq Hom(im(g), SC[-1])$. As $C, SC[-1]$ are $\sigma$-stable by Lemma~\ref{lem:restrictC}, we apply Lemma~\ref{lem:sbound} and obtain $0 \leq\varphi(SC[-1]) - \varphi(C) \leq \frac{1}{5}$, and hence $SC[-1] \in \mathcal{A}'$. The claim follows by applying $Hom(\cdot, SC[-1])$ to the surjection $E \twoheadrightarrow im(g)$ in sequence~\eqref{eq:exact5}. 

Altogether, we obtain the inequalities
\begin{align*}
hom^1(C,C) &= hom^1(C, SE[-1]) - hom^1(C,E) + hom(C,C) \\
&\leq (2d -2) - d + hom(C,C)\\
&= d - 2 + hom(C,C)
\end{align*}
where the first follows by applying $Hom(C, \cdot)$ to sequence~\eqref{eq:exact6} and the vanishings $Hom(C,SE[-1])= 0$ and $Hom^2(C,E) = 0$ which follow from the proof of Lemma~\ref{lem:mainbound}. The second follows from Lemma~\ref{lem:mainbound} applied to sequence~\ref{eq:exact6} and the above paragraph. On the other hand, as $C$ is $\sigma$-stable, we obtain $hom^1(C,C) \leq d-1$, contradicting the minimality of $d$. Thus, it must be the case that $d = 1$.

Finally, assume $F \in \mathcal{T}$ satisfies $hom^1(F,F) = d =1$. By Lemma~\ref{lem:1cy}, the objects $F, SF[-1]$ are $\sigma$-semistable. By Lemma~\ref{lem:65noncurve}, there exists a heart $\mathcal{A}$ with $F, SF[-1] \in \mathcal{A}$ satisfying Definition~\ref{defn:noncurve}. By Lemma~\ref{lem:1spherical}, $F$ must be a $1$-spherical object and we conclude.
\end{proof}

\section{Reconstruction from a $1$-Calabi-Yau object}\label{1cyreconstruction}
Let $\mathcal{D}$ be a triangulated category of finite type over $k$. We recall that $\mathcal{D}$ is a {\em geometric non-commutative scheme} if there exists an admissible embedding $\mathcal{D} \xhookrightarrow{} D^b(X)$ with $X$ a smooth, projective variety over $k$. In this subsection, we prove the following reconstruction result.
\begin{prop}\label{prop:reconstructionA}
Assume that $\mathcal{D}$ is a non-rational, geometric noncommutative scheme with a numerical stability condition $\sigma = (\mathcal{A},Z)$ such that $gldim(\sigma) <2$ and $Im(Z) \subset \mathbb{C}$ is discrete. If there exists an object $E \in \mathcal{D}$ such that $Hom^1(E,E) = k$, then there exists an admissible embedding $\Phi \colon D^b(C) \xhookrightarrow{} \mathcal{D}$ where $C$ is a smooth projective curve.

%If instead, there exists a $1$-Calabi-Yau object and $\sigma$ is Serre-invariant, then the same conclusion holds. In this case, if $\mathcal{D}$ is also connected, then $\Phi$ is an equivalence of categories.
In addition, if $\mathcal{D}$ is connected and if every $\sigma$-stable object $E$ satisfying $Hom^1(E,E) = k$ is $1$-spherical, then $\Phi$ is an equivalence of categories.
\end{prop}
In this subsection, we will always assume $\mathcal{D} \xhookrightarrow{} D^b(X)$ a geometric noncommutative scheme and that all Bridgeland stability conditions are numerical, i.e. that the central charge factorizes through $K_{num}(\mathcal{D})$, the numerical Grothendieck group. We first recall the basic definitions and notions of the moduli spaces of objects in $\mathcal{D}$ following \cite[Section 5.3]{macri2019lectures} and we defer to \cite[Part II]{2021} for the analogous definitions in greater generality and for the fundamental properties.
%\begin{lemma}\label{lem:geomreconstruction}
%Assume that the assumptions of Lemma~\ref{lem:conditionsreconstruct} hold on $\mathcal{D}$. Then there exists an equivalence $\mathcal{D} \simeq D^b(C)$ with $C$ a smooth projective curve.
%\end{lemma}

Given a scheme $B$, locally of finite type over $k$, we define
\[
\mathcal{D}_{qcoh} \boxtimes D_{qcoh}(B) \subset D_{qcoh}(X\times B)
\]
to be the smallest triangulated subcategory in the unbounded derived category of quasi-coherent sheaves on $X\times B$ closed under direct sums and containing $\mathcal{D} \boxtimes D^b(B)$. 
\begin{defn}
An object $E\in D_{qcoh}(X\times B)$ is \textit{B-perfect} if, locally over $B$, it is isomorphic to a bounded complex of quasi-coherent sheaves on $B$ that is flat and of finite presentation.
\end{defn}
\noindent
Let $D_{B-perf}(X \times B)$ be the full subcategory of $B$-perfect complexes in $D_{qcoh}(X\times B)$ and consider the following restriction to $\mathcal{D}$.
\[
\mathcal{D}_{B-perf} \coloneqq (\mathcal{D}_{qcoh} \boxtimes D_{qcoh}(B)) \cap D_{B-perf}(X \times B)
\]
Let $\mathcal{M} \colon (Sch/k)^{op} \rightarrow Grp$ be the $2$-functor defined as follows for $B$, locally of finite type.
\[
\mathcal{M}(B) \coloneqq \bigg\{ E\in \mathcal{D}_{B-perf} \colon \begin{array}{l} Ext^i(E|_{X \times \{b \}}, E|_{X \times \{b\}}) = 0,\text{ for all }i < 0 \\ \text{ and all geometric points }b \in B \end{array}\bigg\}
\]
Given a stability condition $\sigma \in \mathcal{D}$, we denote by $\mathcal{M}_\sigma(v)$ (resp. $\mathcal{M}_\sigma^{st}(v)$) the substack of $\mathcal{M}$ parametrizing $\sigma$-semistable (resp. $\sigma$-stable) objects in $\mathcal{D}$ of class $v$.

The following Theorem summarizes the relevant properties that we will need for these objects.
%We denote by $\mathcal{M}_{pug}(\mathcal{D}) \colon (Sch/k)^{op} \rightarrow Grp$ the functor where
%\[
%\mathcal{M}_{pug}(\mathcal{D})(T) = \{ E \in D_{pug}(X_T /T) \colon E_t \in \mathcal{D}_t\text{ for all }t \in T \}
%\]
%where $\mathcal{D}_t$ denotes the base change category as defined in \cite{Kuznetsov_2011}. 
\begin{theorem}[{\cite[Theorem 21.24]{2021}}, {\cite[Theorem A.5]{muk87}}]\label{thm:stack}
Fix a class $v \in K_{num}(\mathcal{D})$ such that $\mathcal{M}_\sigma(v) = \mathcal{M}_\sigma^{st}(v)$. Then $\mathcal{M}_\sigma(v)$ is an algebraic stack, the coarse moduli space $M_\sigma(v)$ is a proper algebraic space, and $\mathcal{M}_\sigma(v)$ is a $\mathbb{G}_m$-gerbe over $M_\sigma(v)$. In particular, there exists a quasi-universal family in $\mathcal{D}_{B-perf}$. 
\end{theorem}
The critical ingredient in the proof of Proposition~\ref{prop:reconstructionA} is to identify a suitable curve in the moduli space of $\sigma$-semistable objects in $\mathcal{D}$. By imposing a bound on the global dimension of the stability condition, we see that $M_\sigma(v)$ is particularly simple.
\begin{corollary}\label{cor:mprojcurve}
Assume that $\mathcal{D}$ admits a stability condition $\sigma$ such that $gldim(\sigma) < 2$. Assume that there exists a $\sigma$-semistable object $E$ of class $v$ such that $v^2 = 0$ and $\mathcal{M}_\sigma(v) = \mathcal{M}_\sigma^{st}(v)$. Then $M_\sigma(v)$ is a smooth projective curve. 
\end{corollary}
\begin{proof}
By assumption, we have that $\chi(v,v) = 0$. As every $\sigma$-semistable object of class $v$ is $\sigma$-stable by assumption, we have that $Hom(F,F[i]) = k$ for $i = 0,1$ and vanishes otherwise for any $\sigma$-semistable object $F$ of class $v$ from the assumption that $gldim(\sigma) < 2$.

The above implies in particular, that $Hom(F,F[1]) = k$ and $Hom(F,F[2]) = 0$ and hence $M_\sigma(v)$ is of dimension $1$ and is smooth, respectively. Thus, $M_\sigma(v)$ is a smooth and proper algebraic curve, and so is automatically projective by \cite[\href{https://stacks.math.columbia.edu/tag/0A26}{Tag 0A26}]{stacks-project}.
\end{proof}
In order to fully utilize Theorem~\ref{thm:stack}, we will need to find a suitable class $v \in K_{num}(\mathcal{D})$ such that any $\sigma$-semistable object of class $v$ is also $\sigma$-stable. Under mild conditions on the stability condition, this is indeed the case.
\begin{lemma}\label{lem:conditionsreconstruct}
Assume that there exists a stability condition $\sigma = (\mathcal{A}, Z)\in Stab(\mathcal{D})$ such that $gldim(\sigma) < 2$ and $Im(Z) \subset \mathbb{C}$ is discrete. Assume that there exists a $\sigma$-semistable object $E$ such that $v(E)^2 = 0$. Then there exists a class $v \in K_{num}(\mathcal{D})$ such that $v^2 = 0$ and $\mathcal{M}_\sigma(v) = \mathcal{M}_\sigma^{st}(v) \neq 0$.
\end{lemma}
\begin{proof}
By assumption, there exists an object $E \in \mathcal{D}$ which is $\sigma$-semistable and satisfies $v(E)^2 = 0$. If there does not exist a strictly $\sigma$-semistable object of class $v(E)$, then we are done. Otherwise, assume $F$ is strictly $\sigma$-semistable of the same class. Taking a Jordan-H\"{o}lder filtration, we obtain $\sigma$-stable factors $F_i$ such that $\sum_i v(F_i) = v(F)$ and satisfying $v(F_i)^2 = 0$ by Lemma~\ref{lem:semistable}(2). Fixing a class $v_i = v(F_i)$ for some $i$, we proceed by induction, and this must terminate by the assumption of discreteness. 
\end{proof}
Finally, we combine the above observations to deduce Proposition~\ref{prop:reconstructionA}. The following proof largely follows the methods exhibited in \cite[Lemma 32.5]{2021}.
\begin{proof}[Proof of Proposition~\ref{prop:reconstructionA}]
Fix a stability condition $\sigma = (\mathcal{A}, Z) \in Stab(\mathcal{D})$ such that $Im(Z) \subset \mathbb{C}$ is discrete and such that $gldim(\sigma) < 2$. If there exists an object $E\in \mathcal{D}$ such that $Hom^1(E,E) = k$, then by Lemma~\ref{lem:1cy}(3), $E$ must be $\sigma$-semistable and $v(E)^2 = 0$. %On the other hand, if there exists a $1$-Calabi-Yau object and $\sigma$ is Serre-invariant, then by uniqueness of the Harder-Narasimhan filtration and Serre-invariance, every $\sigma$-semistable factor must be $1$-Calabi-Yau.
By Lemma~\ref{lem:conditionsreconstruct}, we may fix a class $v \in K_{num}(\mathcal{D})$ such that $\mathcal{M}_\sigma (v) = \mathcal{M}_\sigma^{st}(v) \neq 0$ and $v^2 = 0$.

By restricting to a connected component of the moduli functor, we may assume that $M_\sigma(v)$ is connected. By corollary~\ref{cor:mprojcurve}, $C \coloneqq M_\sigma(v)$ is a smooth, projective curve. By Theorem~\ref{thm:stack}, there exists a quasi-universal family $\mathcal{E} \in \mathcal{D}_{C-perf} \subset D^b(C \times X)$ which must be universal as there does not exist a non-trivial Brauer class in dimension $1$. Given two non-isomorphic $\sigma$-stable objects $E,F \in \mathcal{D}$ such that $v(E) = v(F) = v$, we have $Hom(E,F) = 0 = Hom(F,E)$ and thus $Hom^i(E,F) = 0 = Hom^i(F,E)$ for all $i$ from the fact that $v^2 = 0$ and the assumption on the global dimension. Thus, by \cite{https://doi.org/10.48550/arxiv.alg-geom/9506012}, the integral transform $\Phi_\mathcal{E} \colon D^b(C) \xhookrightarrow{} D^b(X)$ is an admissible embedding, and moreover factors through $\mathcal{D}$ by definition.

%In the latter case, if $\sigma$ is Serre-invariant, then let $\varphi$ be the phase corresponding to the image $Z(v) \in \mathbb{C}$. By assumption, there exists a $\sigma$-semistable $1$-Calabi-Yau object $E$ contained in $\mathcal{P}_\sigma(\varphi)$. As $SE[-1] \simeq E$, we have that $S\mathcal{P}_\sigma(\varphi)[-1] = \mathcal{P}_\sigma(\varphi)$ by Serre-invariance of $\sigma$. Arguing as in the second paragraph of Proposition~\ref{prop:1cy}, it follows that any $\sigma$-stable object of phase $\varphi$ is also $1$-Calabi-Yau. 
Under the additional assumption, every $\sigma$-stable object of class $v$ is also $1$-Calabi-Yau. Thus, the universal family $\mathcal{E} \in D^b(C \times X)$ is a family of $1$-Calabi-Yau objects in $\mathcal{D}$ and so $\Phi_\mathcal{E}$ is essentially surjective by the same argument as in the second paragraph of the proof of \cite[Theorem 5.4]{bridgeland2019equivalences}.
\end{proof}

\section{Reconstruction of non-commutative curves}\label{sec:main}
In this section, we prove our main Theorem~\ref{thm:mainthm} and deduce a number of additional corollaries for categories of dimension $1$. For convenience of notation, let $Stab_\mathcal{N}(\mathcal{D})$ the space of numerical Bridgeland stability conditions and let $Stab_{\mathcal{N},d}(\mathcal{D})$ be the subspace of stability conditions $\sigma = (\mathcal{A},Z)$ such that $Im(Z) \subset \mathbb{C}$ is discrete.
\begin{theorem}\label{thm:mainthm}
Assume that $\mathcal{D}$ is a connected, non-rational, geometric non-commutative scheme, and that $\inf\limits_{\sigma \in Stab_{\mathcal{N},d}(\mathcal{D})}gldim(\sigma) < \frac{6}{5}$. 
Then there exists an equivalence $\mathcal{D} \simeq D^b(C)$ with $C$ a smooth projective curve of genus $g \geq1$.
\end{theorem}
\begin{proof}
By Lemma~\ref{lem:1cy} and Proposition~\ref{prop:1sphericalnew}, there exists a $1$-spherical object and every object $E \in \mathcal{D}$ satisfying $hom^1(E,E) = 1$ is also $1$-spherical. The conclusion then follows by Proposition~\ref{prop:reconstructionA}.
%For the claim, let $E$ be such an object. Then arguing as in the proof of Proposition~\ref{prop:dimbound}, there exists a positive integer $n$ and a stability condition $\sigma$ such that $S^nE$ satisfies Definition~\ref{defn:noncurve}. By Lemma~\ref{lem:1spherical}, it must be the case that $S^nE$ is $1$-spherical, and so $E$ is also $1$-spherical and the claim follows.
%By Proposition~\ref{prop:reconstructionS}, there exists a slicing $\mathcal{A}_\varphi$ of homological dimension $1$. In particular, this implies that the stability condition $\sigma$ must satisfy $gldim(\sigma) < 2$. By the second part of Proposition~\ref{prop:reconstructionA}, we have that $\mathcal{D} \simeq D^b(C)$ with $C$ a smooth projective curve. As there does not exist any exceptional objects, it must be the case that the genus $g$ of $C$ satisfies $g \geq 1$.
\end{proof}
\begin{rem}\label{rem:infdim}
We note that the assumptions in Theorem~\ref{thm:mainthm} can be weakened in various ways. In particular, the existence of a $1$-spherical object only needs the condition $\inf\limits_{\sigma\in Stab(\mathcal{D})}gldim(\sigma) < \frac{6}{5}$ where the infimum ranges over the full space of stability conditions. Then, we only need a separate stability condition $\sigma$ which in addition is numerical, has discrete central charge and satisfies $gldim(\sigma) < 2$ to achieve the desired equivalence.
\end{rem}
In the case of dimension $1$, we recall the following result.
\begin{theorem}\cite[Theorem 5.16]{Kikuta_2021}\label{thm:kikuta}
Let $C$ be a smooth projective curve of genus $g$.
\begin{enumerate}
\item
If $g = 0$, then there exists a stability condition on $D^b(C)$ such that $gldim(\sigma) = 1$.
\item
If $g = 1$, then $gldim(\sigma) = 1$ for any stability condition $\sigma \in Stab_\mathcal{N}(D^b(C))$. 
\item
If $g \geq 2$, then $gldim(\sigma) > 1$ for any stability condition $\sigma \in Stab_\mathcal{N}(D^b(C))$ and\\ $\inf\limits_{\sigma \in Stab_\mathcal{N}(D^b(C))}gldim(\sigma) = 1$.
\end{enumerate}
\end{theorem}
As a consequence, we obtain the following corollary.
\begin{corollary}\label{cor:nonexistence}
There exists no connected, non-rational, geometric non-commutative schemes $\mathcal{D}$ satisfying $\inf\limits_{\sigma\in Stab_{\mathcal{N},d}(\mathcal{D})}gldim(\sigma) \in( 1, \frac{6}{5})$.
\end{corollary}
\subsection{Serre-invariance and higher genus curves}
As a consequence of Theorem~\ref{thm:mainthm}, we obtain two sharper categorical characterizations of higher genus curves, and a partial converse to Theorem~\ref{thm:kikuta}. In this section, we relate our results to the weaker notion (see \cite[Theorem 4.2]{Kikuta_2021}) of the upper Serre dimension of $\mathcal{T}$ under the additional assumption that the stability condition $\sigma$ is invariant under the action of the Serre functor.

We first recall the notion of a Serre-invariant stability condition and the Serre dimension on $\mathcal{T}$, a triangulated category of finite type over $k$ with a Serre functor $S$ and a classical generator $G$. 
\begin{defn}
Let $\sigma = (\mathcal{A},Z)$ be a Bridgeland stability condition on $\mathcal{T}$. We say that $\sigma$ is Serre-invariant if $S \cdot \sigma = \sigma \cdot g$ for some $g \in \widetilde{GL}_2^+(\mathbb{R})$.
\end{defn}
More concretely, we have the following characterization of the $\widetilde{GL}_2^+(\mathbb{R})$ action.
\begin{rem}[\cite{2002math.....12237B},Lemma 8.2]\label{rem:faction}
Recall that specifying an element $g \in \widetilde{GL}_2^+(\mathbb{R})$ is equivalent to specifying a pair $(T,f)$ with $f \colon \mathbb{R} \rightarrow \mathbb{R}$ an increasing map such that $f (\varphi + 1) = f(\varphi) + 1$ and $T \colon \mathbb{R}^2 \rightarrow \mathbb{R}^2$ an orientation-preserving linear isomorphism agreeing on $\mathbb{R}/2\mathbb{Z}$.
\end{rem}
We also introduce the notion of the upper and lower Serre dimensions and we defer to \cite[Section 5]{2019arXiv190109461E} for the relevant definitions in greater generality. In the presence of a stability condition, the Serre dimension of $\mathcal{T}$ is particularly simple.
\begin{defn}[\cite{Kikuta_2021}, Proposition 3.9]\label{lem:asymptoticphase}
The upper and lower Serre dimension of $\mathcal{T}$ are defined respectively as follows.
\begin{enumerate}
\item
$\overline{Sdim}(\mathcal{T}) \coloneqq \limsup\limits_{n \rightarrow \infty}\frac{1}{n}\varphi_\sigma^+(S^nG)$
\item
$\underline{Sdim}(\mathcal{T}) \coloneqq \limsup\limits_{n \rightarrow \infty}\frac{1}{n}\varphi_\sigma^-(S^nG)$
\end{enumerate}
\end{defn}
\begin{rem}
The fact that these definitions are independent of the classical generator $G$ can be deduced from \cite[Definition 5.3]{2019arXiv190109461E}.
\end{rem}

In this subsection, we prove the following.
\begin{theorem}\label{thm:infdim1}
Assume that $\mathcal{D}$ is a connected, non-rational, geometric non-commutative scheme. Assume that there exist a numerical stability condition $\sigma \in Stab(\mathcal{D})$ such that $gldim(\sigma) < 2$ and $Im(Z) \subset \mathbb{C}$ is discrete. Then the following are equivalent:
\begin{enumerate}
\item
$\inf\limits_{\sigma \in Stab(\mathcal{D})}gldim(\sigma) = 1$
\item
$\mathcal{D} \simeq D^b(C)$ with $C$ a smooth projective curve of genus $g \geq 1$.
\setcounter{serreinv}{\value{enumi}}
\end{enumerate}
If in addition, $\mathcal{D}$ admits a Serre-invariant Bridgeland stability condition, then the above is equivalent to the following.
\begin{enumerate}
\setcounter{enumi}{\value{serreinv}}
\item
$\overline{Sdim}(\mathcal{D}) = 1$
\end{enumerate}
\end{theorem}
In the following, $\mathcal{T}$ will always denote a triangulated category of finite type over $k$ with a Serre functor $S$, admitting a classical generator $G$ and we fix a Bridgeland stability condition $\sigma = (\mathcal{A},Z)$. We first note, that in the presence of an appropriate $\sigma$-semistable object, we can bound the asymptotic phases of such an object under iterated applications of the Serre functor.
\begin{lemma}\label{lem:stablephase}
Assume there exists an object $E \in \mathcal{T}$ such that $S^nE$ is $\sigma$-semistable for all $n \in \mathbb{Z}$. Then we have the following inequalities:
\[
\underline{Sdim}(\mathcal{T}) \leq \limsup\limits_{n \rightarrow \infty} \frac{1}{n} \varphi(S^nE) \leq \overline{Sdim}(\mathcal{T})
\]
\end{lemma}
\begin{proof}
Fix $G \in \mathcal{T}$ to be any classical generator. Then, there must exist integers $i,j$ such that $Hom(G, E[i]), Hom(E, G[j]) \neq 0$. In particular, these imply the following inequalities.
\begin{align*}
\varphi^-(S^nG) \leq \varphi(S^nE) + i \\
\varphi(S^nE) \leq \varphi^+(S^nG) + j
\end{align*}
for all $n$. Applying $\limsup\limits_{n \rightarrow \infty} \frac{1}{n}$ together with Definition~\ref{lem:asymptoticphase} implies the claim.
\end{proof}

In the presence of a Serre-invariant stability condition and under the assumption that the Serre dimension is a positive integer $n$, we can identify a heart of a bounded t-structure $\mathcal{A} \subset \mathcal{T}$ of homological dimension $n$ satisfying Serre duality.

\begin{prop}\label{prop:reconstructionS}
Assume the following:
\begin{itemize}
\item
Assume that $\sigma = (\mathcal{P},Z)$ is Serre-invariant on $\mathcal{T}$.
\item
$\overline{Sdim}(\mathcal{T}) = n$ for $n\in \mathbb{Z}_+$.
\end{itemize}
Then there exists a slicing $\mathcal{A}_\varphi = \mathcal{P}([\varphi,\varphi+1))$ such that $S\mathcal{A}_\varphi[-n] = \mathcal{A}_\varphi$. In addition, $\mathcal{A}_\varphi$ is of homological dimension $n$.
\end{prop}
To prove the above Proposition, we will study the asymptotic action of $S$ on the phases $\varphi$ given by the action described in Remark~\ref{rem:faction}.
\begin{lemma}\label{lem:fixedphase}
Under the assumptions of Proposition~\ref{prop:reconstructionS}, let $(G,f) \in \widetilde{GL}_2^+(\mathbb{R})$ correspond to the action of $S$ on $\sigma = (\mathcal{P}.Z)$. Then there exists a real number $\varphi \in \mathbb{R}$ such that $f(\varphi) = \varphi + n$.
\end{lemma}
\begin{proof}
By assumption, the action of the auto-equivalence $F \coloneqq S[-n]$ on $\sigma$ is equivalent to an action by an element of $\widetilde{GL}_2^+(\mathbb{R})$ with the phase function $f_F(\varphi) = f(\varphi) - n$. Similarly, we denote by $f_{F^{-1}}$ the phase function associated with the inverse auto-equivalence $F^{-1} = S^{-1}[n]$. We will prove that the function $f_F \colon \mathbb{R} \rightarrow \mathbb{R}$ has a fixed point. 

Fix a phase $\varphi_0$ such that $\mathcal{P}(\varphi_0) \neq 0 $. By \cite[Proposition 6.17]{kuznetsov2021serre} and Lemma~\ref{lem:stablephase}, we must have $\limsup\limits_{k\rightarrow \infty} \frac{1}{k} f^k_F(\varphi_0) = 0$. We first prove that the set $\{f_F^k(\varphi_0)\}_{k \in \mathbb{Z}_+} \subset \mathbb{R}$ is bounded of length $<1$. Assume on the contrary that the set $\{f_F^k(\varphi_0)\}_{k\in\mathbb{Z}_+}$ is of length $\geq 1$. Then acting by $f_{F^{-1}}$, which is increasing and periodic of order $1$, there must exist $m \in \mathbb{Z}_+$ such that either $f_F^m(\varphi_0) \geq \varphi_0 + 1$ or $\varphi_0 -1 \geq f_F^m(\varphi_0)$. 

We assume the first case. By the assumption that $f_F$ is increasing and periodic of order 1, for any positive integer $k \in \mathbb{Z}$, we have that $f_F^{mk}(\varphi_0) \geq f_F^{m(k-1)}(\varphi_0+1) = f_F^{m(k-1)}(\varphi_0) + 1 \geq \ldots \geq \varphi_0 + k$. But then we have
\[
0 = \limsup\limits_{k\rightarrow\infty}\frac{1}{k} f^k_F(\varphi_0) \geq \limsup\limits_{k\rightarrow\infty}\frac{1}{km}f^{km}_F(\varphi_0) \geq \limsup\limits_{k\rightarrow\infty}\frac{1}{km} (\varphi_0 +k)  = \frac{1}{m}
\]
where the first inequality follows as $\frac{1}{km}f_F^{km}(\varphi_0)$ is a subsequence of $\frac{1}{k}f_F^k(\varphi_0)$ and the second follows from the inequalities of the previous sentence. As $m \in \mathbb{Z}_+$, this is a contradiction.

In the latter case, an identical argument as in Lemma~\ref{lem:stablephase} with the inverse Serre functor together with \cite[Proposition 3.9]{Kikuta_2021} implies that $\limsup\limits_{k\rightarrow \infty} \frac{1}{k} f^k_{F^{-1}}(\varphi_0) = 0$. In particular, we have $f^m_{F^{-1}}(\varphi_0) \geq \varphi_0 + 1$ and an identical argument as in the above paragraph yields a contradiction.

Define $\varphi_0^+ \coloneqq \limsup\limits_{k\rightarrow \infty}f_F^k(\varphi_0)$. We claim that this gives the desired fixed point. Indeed, we have the sequence of equalities 
\[f_F(\limsup\limits_{k\rightarrow \infty}f_F^k(\varphi_0)) = \limsup\limits_{k\rightarrow\infty}f_F^{k+1}(\varphi_0) = \limsup\limits_{k\rightarrow\infty}f_F^k(\varphi_0)
\]
In the first equality above, we have used the elementary fact that given a continuous increasing function $f \colon A \rightarrow \mathbb{R}$ with $A$ compact and a sequence $(x_k)_{k \in \mathbb{Z}_+} \subset A$, we have the equality
\[
f(\limsup\limits_{k \rightarrow \infty}x_k) = \limsup\limits_{k \rightarrow \infty}f(x_k)
\]
Thus, we have the desired fixed point, and we conclude.
\end{proof}
We now turn to proving Proposition~\ref{prop:reconstructionS}.
\begin{proof}[Proof of Proposition~\ref{prop:reconstructionS}]
By Lemma~\ref{lem:fixedphase}, there exists $\varphi_0 \in \mathbb{R}$ such that $S(\varphi_0) = \varphi_0+ n$. As the action is periodic of order one, we have $S(\varphi_0 - 1) = S(\varphi_0) - 1 = \varphi_0 + n-1$. As the action is continuous and increasing, we have that $S((\varphi_0 - 1, \varphi_0]) = (\varphi_0 + n-1, \varphi_0 + n]$. Thus, we conclude that $S(\mathcal{P}((\varphi_0 - 1, \varphi_0])) = \mathcal{P}((\varphi_0 + n - 1, \varphi_0 +n])$ as desired.

We now prove that the heart $\mathcal{A}_\varphi\coloneqq \mathcal{P}(\varphi_0 -1, \varphi_0]$ is indeed of homological dimension $n$. For any objects $A,B \in \mathcal{A}_\varphi$, we have $Hom(A,B[i]) = 0$ for $i < 0 $. On the other hand, we have $Hom(A,B[i]) = Hom(B[i], SA) = Hom(B[i-n],SA[-n]) = Hom^{n-i}(B,SA[-n]) = 0$ for $i > n$ as $SA[-n] \in \mathcal{A}_\varphi$ by Proposition~\ref{prop:reconstructionS}. 
\end{proof}
%We follow the argument of \cite[Lemma 5.15]{Kikuta_2021}
%\begin{lemma}
%Let $\mathcal{T}$ be a triangulated category with a Serre-invariant Bridgeland stability condition $\sigma = (\mathcal{A},Z)$. Assume that the auto-equivalence $S[-1]$ preserves the heart $\mathcal{A}$, i.e. $S\mathcal{A}[-1] = \mathcal{A}$. Then we have the equality $$\overline{Sdim}(\mathcal{T}) = \inf\limits_{\sigma \in Stab(\mathcal{D})}gldim(\sigma) = 1$$ 
%\end{lemma}
%\begin{proof}
%We first note that the global dimension $\sigma$ must satisfy $gldim(\sigma) < 2$. Indeed...

%By \cite[Proposition 4.3]{Kikuta_2021}, it suffices to prove that for any $\epsilon \in (0,1)$, there exists a stability condition $\sigma_\epsilon$ such that $S(\mathcal{P}_{\sigma_\epsilon}(\varphi)) \subset \mathcal{P}_{\sigma_\epsilon}([\varphi+1 - \epsilon, \varphi+1 + \epsilon])$. 
%\end{proof}
Finally, we apply the conclusion of Proposition~\ref{prop:reconstructionS} to deduce Theorem~\ref{thm:infdim1}.

\begin{proof}[Proof of Theorem~\ref{thm:infdim1}]
$(1) \implies(2)$: This follows directly by Theorem~\ref{thm:mainthm} and remark~\ref{rem:infdim}.

$(2) \implies (1)$: This follows directly by Theorem~\ref{thm:kikuta}.

$(1) \implies (3)$: By \cite[Theorem 4.2]{Kikuta_2021} and the assumption, we have that $\overline{Sdim}(\mathcal{D}) \leq 1$. Assume on the contrary that $\overline{Sdim}(\mathcal{D}) < 1$. By Proposition~\ref{prop:1sphericalnew}, there exists an object $A$ such that $hom^1(A,A) = 1$, and hence $S^nA$ is $\sigma$-semistable for all $n$ by Lemma~\ref{lem:1cy}(3). By Lemma~\ref{lem:stablephase}, we have that $\limsup\limits_{n \rightarrow \infty} \frac{1}{n} \varphi(S^nA) < 1$. In particular, there must exist an integer $N$ such that $\varphi(S^NA) - \varphi(S^{N-1}A) < 1$. On the other hand, we have $Hom(A,A[n]) = Hom(S^{N-1}A,S^{N-1}A[n]) = Hom(S^{N-1}A[n],S^NA) = 0$ for $n \neq 0$ by the assumption on the phases and $\sigma$-semistability of $A$. Thus, $A$ must be an exceptional object, contradicting the assumption.

$(3) \implies (1)$: By Proposition~\ref{prop:reconstructionS}, there exists a slicing $\mathcal{A}_\varphi$ of homological dimension $1$. By \cite[Theorem 2]{ctx31462736420006531}, there exists a $1$-spherical object in $\mathcal{A}_\varphi$.  By \cite[Lemma 3.3]{ctx31462736420006531} and the second part of Proposition~\ref{prop:reconstructionA}, we have that $\mathcal{D} \simeq D^b(C)$ with $C$ a smooth projective curve. As there does not exist any exceptional objects, it must be the case that the genus $g$ of $C$ satisfies $g \geq 1$.
\end{proof}

\subsection{The case of $\bm{gldim(\sigma) = 1}$}\label{gldim1}
In this subsection, we apply Theorem~\ref{thm:mainthm} to deduce a structural result for $\mathcal{D}$ when there exists a stability condition $\sigma$ with global dimension $gldim(\sigma) = 1$. In particular, our conclusion applies in the setting when there exists exceptional objects in $\mathcal{D}$. 
%In this subsection, we prove Theorem~\ref{thm:gldim1} and corollary~\ref{cor:gldim1}. In the following $\mathcal{D}$ will always be a geometric noncommutative scheme with a Serre functor $S$. We will assume that all Bridgeland stability conditions are numerical.
%\begin{theorem}\label{thm:gldim1}
%Assume that $\mathcal{D}$ is a connected, non-rational noncommutative curve. Assume in particular, that there exists a stability condition $\sigma \in Stab(\mathcal{D})$ such that $gldim(\sigma) = 1$ and that $Im(Z) \subset \mathbb{C}$ is discrete. Then $\mathcal{D} \simeq D^b(E)$ with $E$ a smooth elliptic curve.
%\end{theorem}
\begin{corollary}\label{cor:gldim1}
Assume that $\mathcal{D}$ is a connected, geometric non-commutative scheme. Assume that there exists a numerical Bridgeland stability condition $\sigma \in Stab(\mathcal{D})$ such that $Im(Z) \subset \mathbb{C}$ is discrete and such that $gldim(\sigma) = 1$. Then $\mathcal{D} = \langle \mathcal{C}, E_1 ,\ldots E_n \rangle$ admits a semi-orthogonal decomposition for some integer $n$ with $E_i \in \mathcal{D}$ exceptional objects and $\mathcal{C}$ is either zero or equivalent to $D^b(E)$, with $E$ a smooth elliptic curve.
\end{corollary}
The main additional input is the following lemma, which allows us to inductively study admissible subcategories $\mathcal{D}$ using the theory of stability conditions.
\begin{lemma}\cite[Proposition 5.2]{Kikuta_2021}\label{lem:mon}
Let $\sigma \in Stab(\mathcal{D})$ be a stability condition such that $gldim(\sigma) = 1$ and $\mathcal{D}' \xhookrightarrow{} \mathcal{D}$ a nonzero admissible subcategory. Then there exists a stability condition $\sigma' \in Stab(\mathcal{D}')$ such that $gldim(\sigma') \leq gldim(\sigma)$.
\end{lemma}
We can now easily deduce corollary~\ref{cor:gldim1}.
\begin{proof}[Proof of corollary~\ref{cor:gldim1}]
If $\mathcal{D}$ contained an exceptional object $E$, then we take the orthogonal subcategory $E^\perp$. By Lemma~\ref{lem:mon}, there exists a stability condition $\sigma' \in Stab(E^\perp)$ such that $gldim(\sigma') \leq gldim(\sigma)$. If $gldim(\sigma') < gldim(\sigma)$, then by \cite[Lemma 5.5]{Kikuta_2021}, any object $E' \in E^\perp$ is exceptional if and only if it is $\sigma'$-stable and in particular, there must exist an exceptional object. As $HH_0(\mathcal{D}) < \infty$ because $\mathcal{D}$ is a geometric noncommutative scheme, we proceed by induction and conclude that $\mathcal{D} = \langle E_1 , \ldots , E_n \rangle$ is generated by a full exceptional collection. 

If $gldim(\sigma') = gldim(\sigma) = 1$ and $E^\perp$ contains no exceptional objects, then we conclude by Theorem~\ref{thm:infdim1} together with Theorem~\ref{thm:kikuta}(1). If not, then we take an exceptional object in $E^\perp$ and its orthogonal and continue by induction. This procedure must terminate again by boundedness of the hochschild homology, and the claim follows. 
\end{proof}
\bibliographystyle{amsplain}
\bibliography{bondalserre}
\end{document}